\documentclass{amsart}
\usepackage{amscd}
\usepackage{amssymb,latexsym}

\usepackage[T1]{fontenc}
\usepackage{charter}
\usepackage{eucal}
\usepackage[all,cmtip]{xy}
\usepackage{amsaddr}
\usepackage{graphicx}
\usepackage{arydshln}
\usepackage{verbatim}
\usepackage[pdftex,colorlinks=true]{hyperref}

\theoremstyle{plain}
\newtheorem{thm}{Theorem}[section]
\newtheorem*{main}{Main Theorem}

\newtheorem{lem}[thm]{Lemma}
\newtheorem{prop}[thm]{Proposition}
\newtheorem{cor}[thm]{Corollary}

\theoremstyle{definition}
\newtheorem{defn}[thm]{Definition}

\theoremstyle{remark}
\newtheorem{rem}[thm]{Remark}

\numberwithin{equation}{section}

\newcommand{\wt}{\widetilde}
\newcommand{\wh}{\widehat}
\newcommand{\R}{\mathbb{R}}
\newcommand{\Z}{\mathbb{Z}}
\newcommand{\Or}{\mathcal{O}}
\newcommand{\restrictionmap}[2]{{#1}\mathpunct\restriction\hbox{}_{#2}}
\providecommand{\abs}[1]{\left\lvert#1\right\rvert}

\DeclareMathOperator{\Map}{Map}
\DeclareMathOperator{\Loc}{Loc}
\DeclareMathOperator{\id}{Id}
\DeclareMathOperator{\Iso}{Iso}
\DeclareMathOperator{\cl}{cl}
\DeclareMathOperator{\bd}{bd}
\DeclareMathOperator{\I}{I}

\title{The Hopf type theorem for equivariant gradient local maps} 

\author[P. Bart{\l}omiejczyk and P. Nowak-Przygodzki]
{Piotr Bart{\l}omiejczyk and Piotr Nowak-Przygodzki}
\address{Faculty of Applied Physics and Mathematics,
Gda{\'n}sk University of Technology,
Gabriela Narutowicza 11/12,
80-233 Gda{\'{n}}sk, Poland}
\email{pbartlomiejczyk@mif.pg.gda.pl, piotrnp@wp.pl}
\thanks{The second author was supported by 
Polish Research Grants NCN 2011/03/B/ST1/04533
and 2014/15/B/ST1/01710.}

\date{\today}

\begin{document}

\begin{abstract}
We construct a degree-type otopy invariant for
equivariant gradient local maps in the case
of a real finite dimensional orthogonal
representation of a compact Lie group.
We prove that the invariant establishes
a bijection between the set of
equivariant gradient otopy classes
and the direct sum of 
countably many copies of $\mathbb{Z}$.
\end{abstract}

\maketitle

\section*{Introduction} 
\label{sec:intro}

In the 1920s H. Hopf discovered the homotopy classification
of maps from $n$-dimensional closed oriented manifold 
into the $n$-sphere.
Namely, he proved that the topological degree of the map
determines its homotopy class. More precisely, there is a bijection
from the set of homotopy classes of such maps 
to the integers.

In 1985, to obtain new bifurcation results, E.\ N.\ Dancer (\cite{D})
introduced a new degree-type invariant for $S^1$-equivariant
gradient maps, which provides more information than
the usual equivariant degree. Interestingly, five years later
A.\ Parusi\'nski (\cite{P}) showed that in the absence of group action
such an extra homotopy invariant for gradient maps does not exist.
The idea of the more subtle invariant in the equivariant gradient case
was developed later by S.\ Rybicki and his collaborators.
Many applications of this construction can be found in
\cite{GR,MR,R}.

In this paper we present the Hopf type classification
of the set of otopy classes of equivariant gradient local maps
$\mathcal{F}^\nabla_G[V]$ in the case of
a real finite dimensional orthogonal
representation $V$ of a compact Lie group $G$.
More precisely, first we show a decomposition 
of $\mathcal{F}^\nabla_G[V]$
into factors indexed by orbit types appearing in $V$
and then we give a description of each factor
as the direct sum of a countable number of $\Z$.
It should be emphasized that our result explains the phenomenon
discovered by Dancer. Namely, in the decomposition 
of $\mathcal{F}^\nabla_G[V]$ appear additional factors,
which do not occur in the decomposition 
of the set of otopy classes of equivariant local maps 
$\mathcal{F}_G[V]$. Thus we obtain essentially stronger
otopy invariant in the gradient case.
Recall that otopy is a generalization of the concept
of homotopy introduced by Becker and Gottlieb (\cite{BG1,BG2})
and independently by Dancer, G\k{e}ba and Rybicki (\cite{DGR}).

This study is the natural continuation of our previous work.
Earlier we investigated the classification of otopy classes
in cases of gradient local maps in $\R^n$ (\cite{BP1,BP2}),
gradient local fields on manifolds (\cite{BP4}) and 
equivariant local maps on a representation of 
a compact Lie group (\cite{B1,B2}).
It is worth pointing out that the ideas presented here
were inspired by \cite{B,BK,BKS,BKu,GKW,MP1,MP2}. 
Moreover, our paper develops and clarifies the material contained in \cite{BGI,GI}. 

The arrangement of the paper is as follows. 
Section \ref{sec:prel} contains some preliminaries.
In Sections \ref{sec:phi} and \ref{sec:psi}
we present constructions of the functions
$\Phi$ and $\Psi$, which are essential
to define the invariant $\Theta$.
Section \ref{sec:theta} provides the description of 
the formula for $\Theta$. In Sections \ref{sec:hnormal}
and \ref{sec:onormal} we introduce the notions of
$(H)$-normal and orbit-normal maps.
Our two main results are stated in Section \ref{sec:main},
where also the proof of the second is presented.
In turn, the first result is proved in Section~\ref{sec:proof}.
In Section \ref{sec:parusinski} we show
the Parusi\'nski type theorem,
which establishes the relation between the sets of equivariant
and equivariant gradient otopy classes.
Finally, Section \ref{sec:parametrized} contains
some remarks concerning the parametrized case.

\section{Preliminaries} 
\label{sec:prel}

The notation $A\Subset B$ means 
that $A$ is a compact subset of $B$.
For a topological space $X$,
we denote by $\tau(X)$ the topology on $X$.
For any topological spaces $X$ and $Y$, 
let $\mathcal{M}(X,Y)$ be the set
of all continuous maps $f\colon D_f\to Y$ such that 
$D_f$ is an open subset of $X$. 
Let $\mathcal{R}$ be
a family of subsets of $Y$. We define 
\[
\Loc(X,Y,\mathcal{R}):=\{\,f\in\mathcal{M}(X,Y)\mid 
f^{-1}(R)\Subset D_f \text{ for all $R\in\mathcal{R}$}\,\}.
\]
We introduce a topology in $\Loc(X,Y,\mathcal{R})$
generated by the subbasis consisting of all sets of the form
\begin{itemize}
	\item $H(C,U):=\{\,f\in\Loc(X,Y,\mathcal{R})\mid 
	      C\subset D_f,\, f(C)\subset U\,\}$
	      for $C\Subset X$ and $U\in\tau(Y)$,
	\item $M(V,R):=\{\,f\in\Loc(X,Y,\mathcal{R})\mid 
	      f^{-1}(R)\subset V\,\}$ for $V\in\tau(X)$ and 
	      $R\in\mathcal{R}$.
\end{itemize}
Elements of $\Loc(X,Y,\mathcal{R})$ 
are called \emph{local maps}.
The natural base point of $\Loc(X,Y,\mathcal{R})$ 
is the empty map.
Let $\sqcup$ denote the union of two disjoint local maps.
Moreover, in the case when $\mathcal{R}=\{\{y\}\}$
we will write $\Loc(X,Y,y)$ omitting double curly brackets.

Assume that $V$ is a real finite dimensional orthogonal
representation of a compact Lie group $G$.
Let $X$ be an arbitrary $G$-space.
We say that $f\colon X\to V$ is \emph{equivariant},
if $f(gx)=gf(x)$ for all $x\in X$ and $g\in G$.
We will denote by 
$
\mathcal{F}_G(X)
$
the space $\{f\in\Loc(X,V,0)\mid\text{$f$ is equivariant}\}$
with the induced topology.
Assume that $\Omega$ is an open invariant subset of $V$.
Elements of $\mathcal{F}_G(\Omega)$ are called
\emph{equivariant local maps}.

Let $I=[0,1]$. 
We assume that the action
of $G$ on $I$ is trivial.
Any element of 
$\mathcal{F}_G(I\times\Omega)$
is called an \emph{otopy}.
Each otopy corresponds to a path in $\mathcal{F}_G(\Omega)$ 
and vice versa.
Given an otopy 
$h\colon\Lambda\subset I\times\Omega\to V$ 
we can define for each $t\in I$:
\begin{itemize}
	\item sets $\Lambda_t=\{x\in\Omega\mid(x,t)\in\Lambda\}$,
	\item maps $h_t\colon\Lambda_t\to V$ with $h_t(x)=h(x,t)$.
\end{itemize}
In this situation we say that 
$h_0$ and $h_1$ are \emph{otopic}.
Otopy gives an equivalence relation 
on $\mathcal{F}_G(\Omega)$.
The set of otopy classes will be denoted by 
$\mathcal{F}_G[\Omega]$.

Let 
$
\mathcal{F}^\nabla_G(\Omega)
$
denote the subspace of $\mathcal{F}_G(\Omega)$ 
(with the relative topology)
consisting of those maps $f$ 
for which there is an invariant
$C^1$-function $\varphi\colon D_f\to\R$ 
such that $f=\nabla\varphi$.
We call such maps \emph{gradient}.
Similarly, we write $\mathcal{F}^\nabla_G(I\times\Omega)$
for the subspace of $\mathcal{F}_G(I\times\Omega)$ 
consisting of such otopies $h$ that
$h_t\in\mathcal{F}^\nabla_G(\Omega)$ for each $t\in I$.
These otopies are called \emph{gradient}.
Let us denote by 
$
\mathcal{F}^\nabla_G[\Omega]
$
the set of the equivalence classes 
of the \emph{gradient otopy} relation.

If $H$ is a closed subgroup of $G$ then
\begin{itemize}
\item $(H)$ stands for the conjugacy class of $H$,
\item $NH$ is the normalizer of $H$ in $G$,
\item $WH$ is the Weyl group of $H$ i.e. $WH=NH/H$.
\end{itemize}

Recall that $G_x=\{g\in G\mid gx=x\}$.
We define the following subsets of $V$:
\begin{align*}
V^H &=\{x\in V\mid H\subset G_x\},\\
\Omega_H &=\{x\in\Omega\mid H=G_x\}.
\end{align*}

Set $\Iso(\Omega):=\{(H)\mid
\text{$H$ is a closed subgroup of $G$ and $\Omega_{H}\neq\emptyset$}\}$.
The set $\Iso(\Omega)$ is partially ordered.
Namely, $(H)\le(K)$ if $H$ is conjugate to a subgroup of $K$.
	
We will make use of
the following well-known facts:
\begin{itemize}
	\item $WH$ is a compact Lie group,
	\item $V^H$ is a linear subspace of $V$ and an orthogonal
	      representation of $WH$,
	\item $\Omega_H$ is open in $V^H$,	      
	\item the action of $WH$ on $\Omega_H$ is free,
	\item the set $\Iso(\Omega)$ is finite.		
\end{itemize}

Assume that $M$ is a smooth (i.e., $C^1$) connected
manifold without boundary. Let 
$\mathcal{F}(M)
\subset\Loc\left(M,TM,\{M\}\right)$
denote the space of local vector fields
equipped with the induced topology.

Suppose, in addition, that
$M$ is Riemannian. Then
a local vector field $v$ is called \emph{gradient}
if there is a smooth function $\varphi\colon D_v\to\R$
such that $v=\nabla\varphi$.
In that case
$\mathcal{F}(M)$
contains the subspace
$\mathcal{F}^\nabla(M)$
consisting of gradient  local vector fields.
Any element of 
$\Map\left(I,\mathcal{F}(M)\right)$
is called an \emph{otopy} and any element of 
$\Map\left(I,\mathcal{F}^\nabla(M)\right)$
is called a \emph{gradient otopy}. If two local vector fields
are connected by a (gradient) otopy,
we call them \emph{(gradient) otopic}.
Of course, (gradient) otopy gives an equivalence relation 
on $\mathcal{F}(M)$
($\mathcal{F}^\nabla(M)$).
The sets of the respective equivalence classes
will be denoted by $\mathcal{F}[M]$
and $\mathcal{F}^\nabla[M]$.

\section{Definition of \texorpdfstring{$\Phi$}{Phi}}\label{sec:phi}

Assume
that $V$ is a real finite dimensional orthogonal
representation of a compact Lie group $G$ and
$\Omega$ is an open invariant subset of $V$.
The main goal of this section is to define a function
\[
\Phi\colon\mathcal{F}_G^\nabla[\Omega]\to
\prod_{(H)}\mathcal{F}_{WH}^\nabla\left[\Omega_{H}\right],
\]
where the product is taken over $\Iso(\Omega)$.
Before we get into the details
let us sketch the general idea of the construction
on which our definition is based.
Let $f\in\mathcal{F}_G^\nabla(\Omega)$.
Natural approach suggests to take as a value
of $\Phi([f])$ classes of restrictions
\[
f_H:=\restrictionmap{f}{D_f\cap\Omega_{H}}
\]
for every orbit type $(H)$ in $\Iso(\Omega)$.
Unfortunately, so defined $f_H$ may not be
an element of $\mathcal{F}_{WH}^\nabla\left(\Omega_{H}\right)$,
because the set of zeroes of $f_H$ does not need to be compact.
However, it is compact if $(H)$ is maximal in $\Iso(\Omega)$,
because in that case $\Omega_{(H)}$ is closed in~$\Omega$.
Since it is possible to arrange orbit types in $\Iso(\Omega)$
so that $(H_i)\le(H_j)$ implies $j\le i$ we can define
$\Phi_i([f])$ inductively with respect to that linear order.
Namely, in the first step we set $\Phi_1([f])=\left[f_{H_1}\right]$.
Then, since the set of zeroes of $f$ contained in
$\Omega\setminus\Omega_{(H_1)}$ does not need to be compact,
we perturb $f$ to $f'$ in such a way that $D_{f'}\subset D_{f}$,
$f'$ is otopic to $f$ and the set of zeroes of $f'$ in
$\Omega\setminus\Omega_{(H_1)}$ becomes compact.
Note that now $(H_2)$ is maximal in $\Omega\setminus\Omega_{(H_1)}$.
Hence we put $\Phi_2([f])=\left[f'_{H_2}\right]$
and proceed as before for all subsequent orbit types.
 
The formal definition of $\Phi$ will be divided into four steps.

\subsection{Definition of otopy perturbation 
 \texorpdfstring{from $f$ to $f_{U,\epsilon}$}{}}\label{subsec:def}
Assume that $(H)$ is maximal in $\Iso(\Omega)$.
Let $f=\nabla\varphi\in\mathcal{F}_G^\nabla(\Omega)$.
We will define a map $f_{U,\epsilon}$, 
which is otopic to $f$ and which zeros
contained in $\Omega\setminus\Omega_{(H)}$
are compact. To do that we
choose an open invariant subset $U$ of 
$D_f\cap\Omega_{(H)}$ such that
\[
f^{-1}(0)\cap\Omega_{(H)}\subset
U\subset\cl U\Subset\Omega_{(H)}.
\]
For $\epsilon>0$ we introduce the following notation:
\begin{align*}
U^\epsilon=&\{x+v\mid x\in U, v\in N_x, \abs{v}<\epsilon\},\\
B^\epsilon=&\{x+v\mid x\in \bd U, v\in N_x, \abs{v}\le\epsilon\},
\end{align*}
where $N_x$ denotes $\left(T_x\Omega_{(H)}\right)^\bot$.
Throughout the paper whenever symbols $U^\epsilon$
and $B^\epsilon$ appear we assume tacitly that
the set  $\cl\left(U^\epsilon\right)$ 
is contained in some 
tubular neighbourhood of $\Omega_{(H)}$ in $V$.
Observe that this condition is satisfied for
$\epsilon$ sufficiently small.
Moreover, to define $f_{U,\epsilon}$ we will need
two more assumptions:
\begin{equation}\label{eq:def}
U^{\epsilon}\subset D_f
\quad\text{and}\quad
f^{-1}(0)\cap B^{\epsilon}=\emptyset,
\end{equation}
which are also satisfied for $\epsilon$ small enough.

Let us introduce an auxiliary smooth function
$\mu\colon[0,\epsilon]\to\R$ as in Figure \ref{fig:graphs}.
Now define for each $t\in[0,1]$ 
the function $\mu_t\colon[0,\epsilon]\to\R$
by $\mu_t(s)=t\mu(s)+1-t$.
Using the family of functions $\mu_t$  
we will define for $t\in[0,1]$ a family of maps 
$r_t\colon U^{\epsilon}\to U^{\epsilon}$
by the formula
\[
r_t(x+v)=x+\mu_t(\abs{v})v,
\]
where $x\in U, v\in N_x, \abs{v}<\epsilon$.
Observe that $r_0=\id$ and $r_1(x+v)=x$ for $\abs{v}\le2\epsilon/3$.

\begin{figure}[ht]
\centering
\includegraphics[scale=0.7,trim= 45mm 200mm 60mm 42mm]{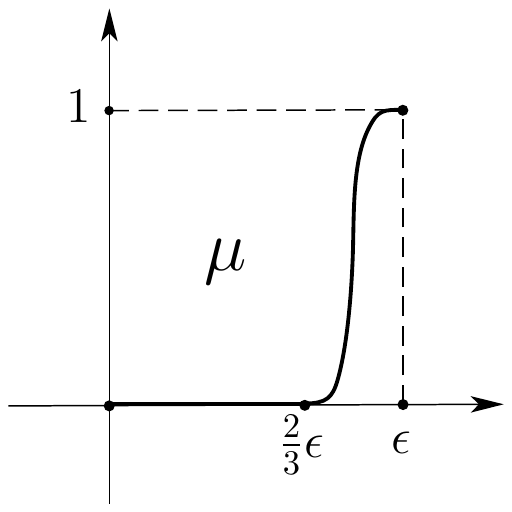}
\includegraphics[scale=0.7,trim= 70mm 200mm 70mm 50mm]{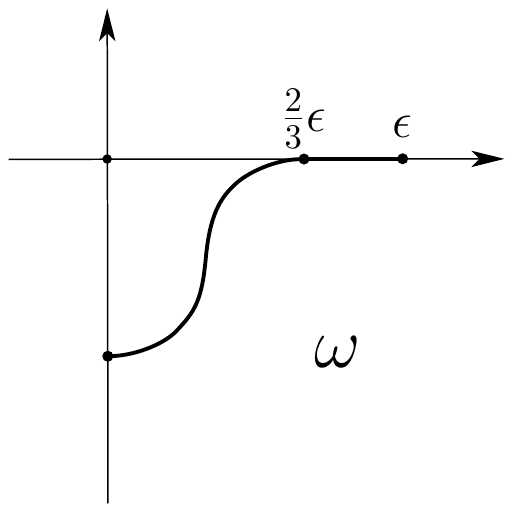}
\caption{Graphs of $\mu$ and $\omega$.}
\label{fig:graphs}
\end{figure}

We will also need another auxiliary function
$\omega\colon[0,\epsilon]\to\R$ (see Figure \ref{fig:graphs}) given by
\[
\omega(s)=
\begin{cases}
\frac12s^2-\frac19\epsilon^2
&\text{for $s\in\left[0,\epsilon/3\right]$},\\
-\frac12\left(s-\frac23\epsilon\right)^2
&\text{for $s\in\left[\epsilon/3,2\epsilon/3\right]$},\\
0
&\text{for $s\in\left[2\epsilon/3,\epsilon\right]$}.
\end{cases}
\]

Now we define for $t\in[0,1]$ a family of potentials
\[
\varphi_t\colon D_f\setminus B^{\epsilon}\to\R
\]
by the formula
\begin{equation}\label{eq:potential}
\varphi_t(z)=\begin{cases}
\varphi(r_{2t}(z))&
\text{if $z\in U^{\epsilon}$ and $t\in[0,1/2]$},\\
\varphi(r_1(z))+(2t-1)\omega(\abs{v})&
\text{if $z=x+v\in U^{\epsilon}$ and $t\in[1/2,1]$},\\
\varphi(z)&
\text{if $z\in D_f\setminus(U^{\epsilon}\cup 
B^{\epsilon})$ and $t\in[0,1]$}.
\end{cases}
\end{equation}
We are now ready to define $f_{U,\epsilon}$
by the formula 
\[
f_{U,\epsilon}=
\nabla\varphi_1.
\]
\begin{figure}[ht]
\centering
\includegraphics[scale=0.47,trim= 80mm 160mm 60mm 55mm]{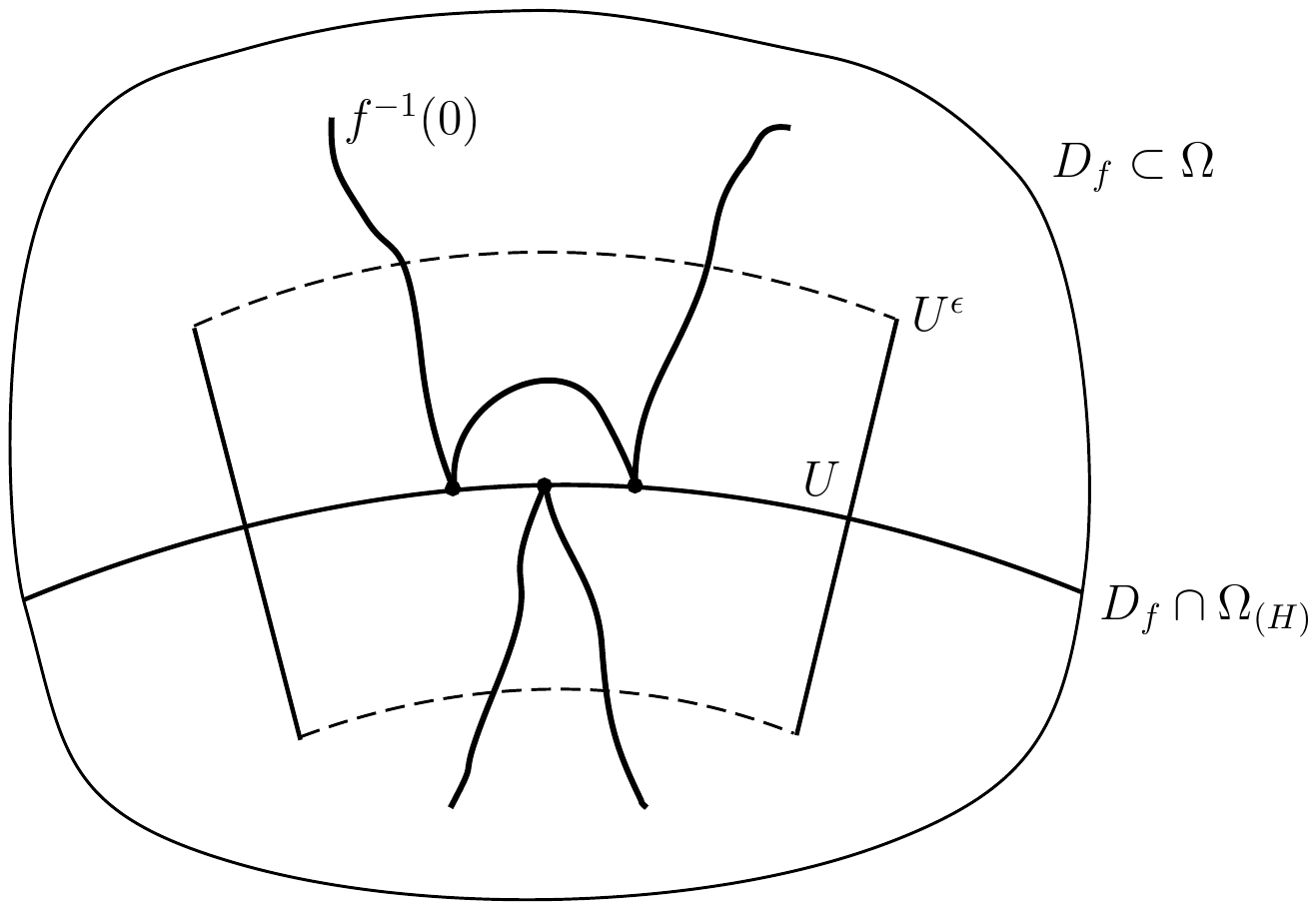}
\includegraphics[scale=0.47,trim= 17mm 160mm 70mm 55mm]{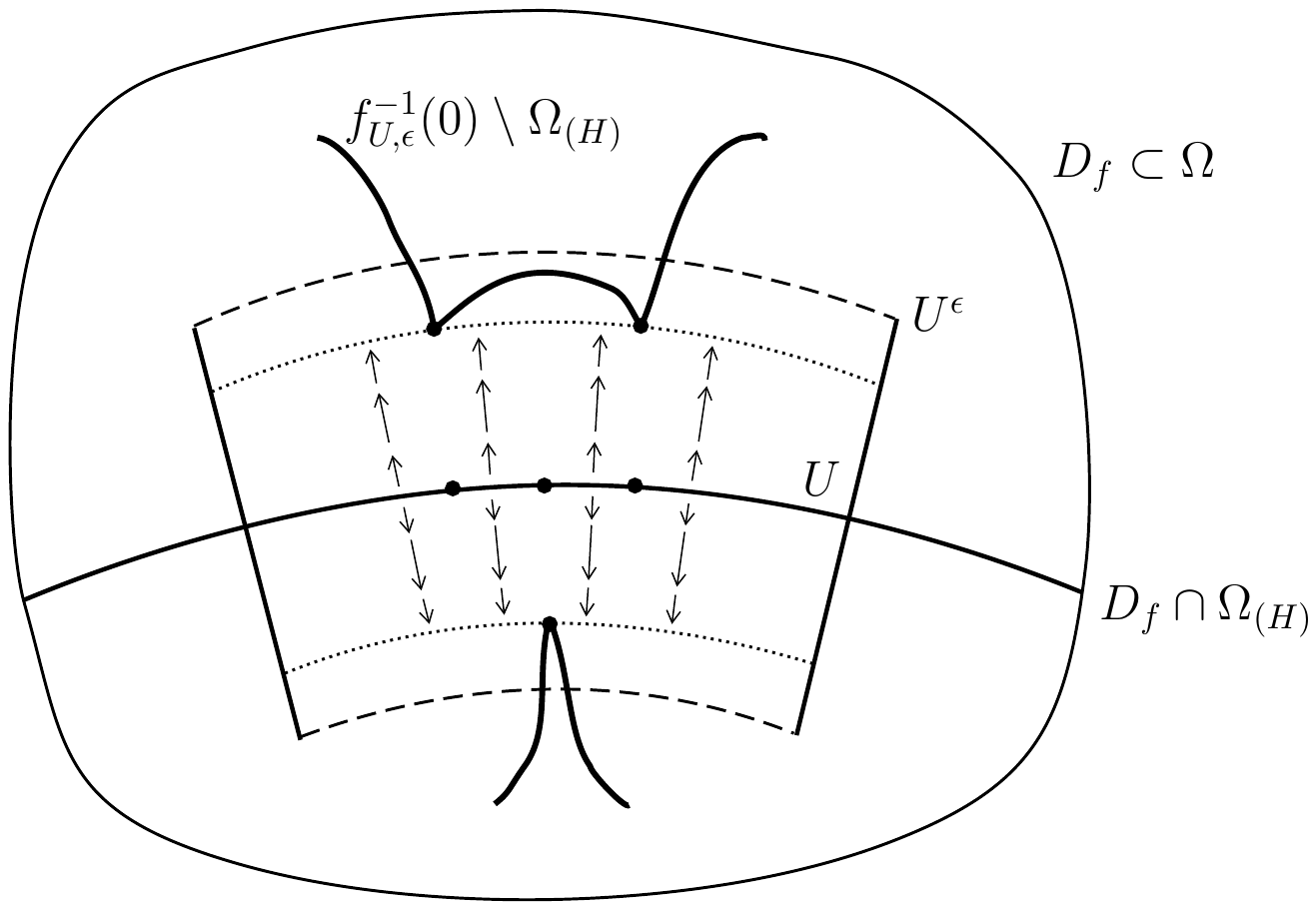}
\caption{Perturbation of $f$.}
\label{fig:perturb}
\end{figure}
\begin{prop}\label{prop:perturb}
The homotopy $h(t,z)=\nabla\varphi_t(z)$ is an otopy
from $\restrictionmap{f}{D_f\setminus B^{\epsilon}}$
to $f_{U,\epsilon}$. 
\end{prop}

\begin{proof}
It is sufficient to show that the set $h^{-1}(0)$ is compact.
Consider the partition of $I\times U^{\epsilon}$
into four subsets (see Figure \ref{fig:part}):
\begin{align*}
A=&[0,1/2]\times U^\epsilon,\\
B=&[1/2,1]\times\left(U^{\epsilon}\setminus U^{2\epsilon/3}\right),\\
C=&[1/2,1]\times\left(U^{2\epsilon/3}\setminus U\right),\\
D=&[1/2,1]\times U.
\end{align*}
\begin{figure}[ht]
\centering
\includegraphics[scale=0.7,trim= 40mm 205mm 60mm 57mm]{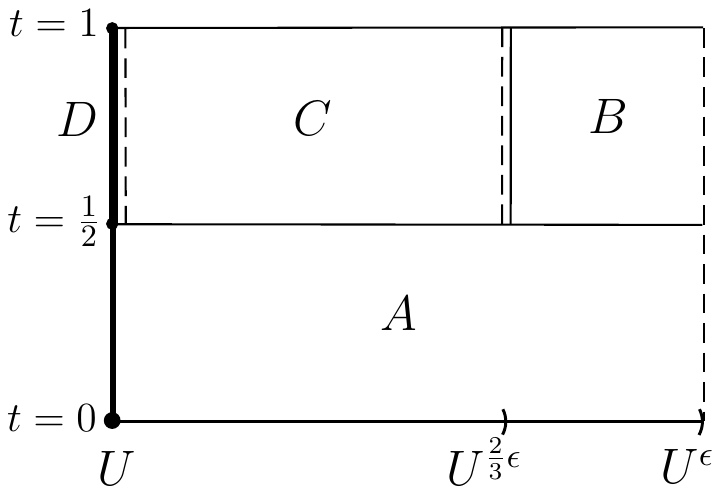}
\caption{Partition of $I\times U^{\epsilon}$.}
\label{fig:part}
\end{figure}

\noindent Since for $t\in[0,1/2]$ and $z\in U^{\epsilon}$ we have
\[
\nabla\varphi_t(z)=
Dr_{2t}^T(z)\cdot
\nabla\varphi_t\left(r_{2t}(z)\right)
\]
and  for $t\in[1/2,1]$ and $z\in U^\epsilon$
the summand $\omega(\abs{v})$
occurs in the formula \eqref{eq:potential},
we obtain the following description 
of the set of zeroes of $h$ in $I\times U^{\epsilon}$:
\begin{itemize}
	\item $h(t,z)=0$ if{f} $f\left(r_{2t}(z)\right)=0$ in $A$,
	\item $h(t,z)=0$ if{f} $f\left(r_{1}(z)\right)=0$ in $B$, 
	\item $h(t,z)\neq0$ in $C$,
	\item $h(t,z)=f(z)$ in $D$.
\end{itemize}
From this we get our claim. 
\end{proof}

Let us introduce the following notation:
\begin{align*}
f^n&=f^n_{U,\epsilon}=\restrictionmap{f_{U,\epsilon}}{U^{\epsilon/3}},\\
f^c&=f^c_{U,\epsilon}=\restrictionmap{f_{U,\epsilon}}{D_{f_{U,\epsilon}}\setminus\Omega_{(H)}},\\
f^a&=\restrictionmap{f^c}{D_{f^c}\setminus\cl\left(U^{\epsilon/3}\right)}.
\end{align*}
\begin{rem}
The following observations will be useful 
in the study of the function~$\wh{\Phi}$,
which will be defined in the next subsection: 
\begin{itemize}
	\item $f^n$ is $(H)$-normal in the sense of Definition \ref{defn:normal},
	\item $\left[f^c\right]=\left[f^a\right]$ in 
	      $\mathcal{F}_G^\nabla\left[\Omega\setminus\Omega_{(H)}\right]$,
	\item $\left[f\right]=\left[f_{U,\epsilon}\right]=\left[f^n\sqcup f^a\right]$ 
	      in $\mathcal{F}_G^\nabla[\Omega]$.
\end{itemize}
\end{rem}

\subsection{Definition of \texorpdfstring{$\wh{\Phi}$}{whPhi}}
Let us define the function
\[
\wh{\Phi}\colon\mathcal{F}_G^\nabla[\Omega]\to
\mathcal{F}_{WH}^\nabla\left[\Omega_H\right]\times
\mathcal{F}_G^\nabla\left[\Omega\setminus\Omega_{(H)}\right]
\] 
by
\[
\wh{\Phi}([f])=(\wh{\Phi}'([f]),\wh{\Phi}''([f]))=
\left(\left[f_H\right],\left[f^c_{U,\epsilon}\right]\right),
\]
where $f_H=\restrictionmap{f}{D_f\cap\Omega_H}$.
The fact that $\wh{\Phi}$ is well-defined will be proved
in Section \ref{sec:proof}.

\subsection{Definition of  \texorpdfstring{$\wh{\Phi}_i$}{whPhii}}
As we have mentioned before the orbit types
are enumerated $(H_1),(H_2),\dotsc,(H_m)$
according to the reverse partial order in  $\Iso(\Omega)$.
Now consider a sequence of open subsets of $\Omega$
\[
\Omega_1\supset\Omega_2\supset\dotsb\supset\Omega_{m},
\]
where $\Omega_1=\Omega$ and 
$\Omega_{i+1}=\Omega_{i}\setminus\Omega_{(H_i)}$.
Let
\[
\wh{\Phi}_i=(\wh{\Phi}'_i,\wh{\Phi}''_i)
\colon\mathcal{F}_G^\nabla[\Omega_{i}]\to
\mathcal{F}_{WH_i}^\nabla\left[\Omega_{H_{i}}\right]\times
\mathcal{F}_G^\nabla\left[\Omega_{i+1}\right]
\] 
be defined as $\wh{\Phi}$ in the previous subsection 
with $\Omega$ replaced by $\Omega_{i}$
and $H$ by $H_i$.

\subsection{Definition of \texorpdfstring{$\Phi$}{Phi}}
Let us start with the inductive definition of
\[
\Xi_i\colon\mathcal{F}_G^\nabla[\Omega]
\to\mathcal{F}_G^\nabla[\Omega_{i+1}].
\]
Set $\Xi_0=\id$ and 
$\Xi_{i}=\wh{\Phi}''_i\circ\Xi_{i-1}$.
Let
\[
\Phi_i\colon\mathcal{F}_G^\nabla[\Omega]\to
\mathcal{F}_{WH_i}^\nabla\left[\Omega_{H_i}\right]
\]
be defined by $\Phi_i=\wh{\Phi}'_i\circ\Xi_{i-1}$.
Finally, let
\[
\Phi\colon\mathcal{F}_G^\nabla[\Omega]\to
\prod_{i=1}^m\mathcal{F}_{WH_i}^\nabla\left[\Omega_{H_i}\right]
\]
be given by $\Phi=(\Phi_1,\dotsc,\Phi_m)$.

\begin{rem}
It is worth pointing out that the obtained $\Phi$
does not depend on the choice of linear extension
of the partial order in $\Iso(\Omega)$, since
$\Omega_{(H)}$ is closed in $\Omega$ for every  
maximal orbit type $(H)$ and therefore the inductive step
can be performed simultaneously 
on all maximal types $(H)$ in $\Omega$.
\end{rem}

\begin{rem}\label{rem:defin}
Alternatively, for given $f\in\mathcal{F}_G^\nabla(\Omega)$ 
we can define two finite sequences
of maps $f_i\in\mathcal{F}_G^\nabla(\Omega_i)$ and 
$f_i'\in\mathcal{F}_{WH_i}^\nabla\left(\Omega_{H_i}\right)$ given by
\[
f_1=f,\quad f_{i+1}=(f_i)^c_{U_i,\epsilon_i},\quad
f_i'=\restrictionmap{f_i}{D_{f_i}\cap\Omega_{H_i}}.
\]
Observe that $D_{f_{i+1}}\subset D_{f_{i}}$
and $\Phi_i([f])=\left[f_i'\right]$.
In this way we obtain the equivalent definition of  $\Phi$.
\end{rem}

\section{Definition of \texorpdfstring{$\Psi$}{Psi}}\label{sec:psi}

The main result of our paper (Main Theorem in Sec. \ref{sec:main})
describes the properties of the function
$\Theta\colon
\mathcal{F}_G^\nabla[\Omega]\to
\prod_{i=1}^m(\sum_{j}\mathbb{Z})$,
which provides a degree type invariant.
In the previous section we have constructed the decomposition
$\Phi\colon\mathcal{F}_G^\nabla[\Omega]\to
\prod_{i=1}^m\mathcal{F}_{WH_i}^\nabla\left[\Omega_{H_i}\right]$.
The function $\Theta$ will be defined as a composition of
the function $\Phi$ and a family of bijections between
the factors $\mathcal{F}_{WH_i}^\nabla\left[\Omega_{H_i}\right]$
and the direct sum of countably many copies of $\mathbb{Z}$.
Below we present the construction of such a bijection.
 
In this section we assume
that $V$ is a real finite dimensional orthogonal
representation of a compact Lie group $G$ ($\dim V>0$),
$\Omega$ is an open invariant subset of $V$,
$G$ acts freely on $\Omega$ and $M:=\Omega/G$.
It is well known that $M$ is a Riemannian manifold
of positive dimension equipped with the so-called quotient Riemannian metric
(see for instance \cite[Prop. 2.28]{GHL}).

If $U$ is an open invariant subset of $\Omega$
and $\varphi\colon U\to\R$ is an invariant function
then $\wt\varphi$ stands for the quotient function
$\wt\varphi\colon U/G\to\R$.
Let the function 
$\Psi\colon\mathcal{F}^\nabla_G[\Omega]
\to\mathcal{F}^\nabla[M]$
be given by $\Psi([\nabla\varphi])=[\nabla\wt\varphi]$.
The following result was proved in \cite[Cor. 5.2]{BP4}.
\begin{thm}\label{thm:psi}
$\Psi$ is a well-defined bijection.
\end{thm}

\begin{rem}\label{rem:gradient}
Let $\{M_j\}$ denote the set of components of $M$.
In \cite{BP4} we proved that the intersection number $\I$
establishes a bijection $\mathcal{F}^\nabla[M_j]\approx\Z$.
Consequently, the restrictions of $f\in\mathcal{F}^\nabla(M)$
to the components of $M$ establish a natural bijection
$\I\colon\mathcal{F}^\nabla[M]\to\sum_j\Z$.
Note that a direct sum (not product) appears in the last formula,
since for any gradient local vector field $f$ the preimage 
of the zero section meets only a finite
number of components of $M$ and, in consequence, almost all
restrictions of $f$ are otopic to the empty map.
\end{rem} 

\begin{cor}\label{cor:psi}
The composition
\[
\I\circ\Psi\colon
\mathcal{F}_G^\nabla[\Omega]\to
\mathcal{F}^\nabla[\Omega/G]\to
\sum_j\Z,
\]
where the direct sum is taken over
the set of connected components of $\Omega/G$,
is a bijection. Moreover, if 
$f,g\in\mathcal{F}_G^\nabla(\Omega)$ such that
$D_f\cap D_g=\emptyset$ then
\[
\I\circ\Psi\left(\left[f\sqcup g\right]\right)=
\I\circ\Psi\left(\left[f\right]\right)+
\I\circ\Psi\left(\left[g\right]\right).
\]
\end{cor}

\section{Definition of  \texorpdfstring{$\Theta$}{Theta}}\label{sec:theta}

We are now ready to define the invariant $\Theta$.
Assume that $V$ is a real finite dimensional orthogonal
representation of a compact Lie group $G$ and
$\Omega$ is an open invariant subset of $V$.
Let
$\Psi_i\colon\mathcal{F}^\nabla_{WH_i}[\Omega_{H_i}]
\to\mathcal{F}^\nabla[\Omega_{H_i}/WH_i]$
denote the function $\Psi$ defined in the previous section
with $G$ replaced by $WH_i$ and
$\Omega$ replaced by $\Omega_{H_i}$
($WH_i$ acts freely on $\Omega_{H_i}$).
Recall that $(\Omega_{H_i}/WH_i)_j$ denotes
the $j$-th component of $\Omega_{H_i}/WH_i$.
Let $\pi_{ij}\colon\mathcal{F}^\nabla[\Omega_{H_i}/WH_i]\to
\mathcal{F}^\nabla[(\Omega_{H_i}/WH_i)_j]$ 
denote the function induced by the restriction 
of a gradient local vector field 
to $j$-th component of $\Omega_{H_i}/WH_i$. Finally,
set 
\[
\Theta_{ij}=\I\circ\,\pi_{ij}\circ\Psi_i\circ\Phi_i
\]
and 
\[
\Theta=\{\Theta_{ij}\}
\colon\mathcal{F}_G^\nabla[\Omega]\to
\prod_{i=1}^m\biggl(\sum_{j}\mathbb{Z}\biggr).
\]
Correctness of the above definition requires 
that $\dim\Omega_{H_i}/WH_i>0$ for $i=1,\dotsc,m$,
which is necessary for calculation of the intersection number.
This is the case when $\dim V^{H_1}>0$.
The opposite case, in which $0\in\Omega$ and $\dim V^G=0$,
will be discussed in Remark \ref{rem:opposite} after the proof of Main Result.

\section{ \texorpdfstring{$(H)$-normal}{(H)-normal} maps}\label{sec:hnormal}

The maps discussed in this section are essential
for the formulation of Theorem \ref{thm:phi}.
These maps are important, because in some sense
they are ``generic'' with respect to $\Phi$.
Namely, if $f$ is $(H_i)$-normal then
$\Phi_i([f])=[f_{H_i}]$,
where $f_{H_i}=\restrictionmap{f}{D_f\cap\Omega_{H_i}}$,
and $\Phi_i([f])=[\emptyset]$ for $i\neq j$.
Let $U$ be an open bounded invariant subset of $\Omega_{(H)}$.
Recall that 
$U^\epsilon=\{x+v\mid x\in U, v\in N_x, \abs{v}<\epsilon\}$
and $B^\epsilon=\{x+v\mid x\in \bd U, v\in N_x, \abs{v}\le\epsilon\}$.

\begin{defn}\label{defn:normal}
A map $f\in\mathcal{F}_G^\nabla(\Omega)$ 
is called \emph{$(H)$-normal} (on $U^\epsilon$) if
\begin{itemize}
\item $f^{-1}(0)\subset U$,
\item $U^\epsilon\subset D_f$,
\item $\restrictionmap{f}{U^\epsilon}=\nabla\varphi$, where
      $\varphi\colon U^\epsilon\to\R$ is 
      an invariant $C^1$-function satisfying \linebreak
      $\varphi(x+v)=\varphi(x)+\frac12\abs{v}^2$     
      for $x\in U$, $v\in N_x\Omega_{(H)}$, $\abs{v}<\epsilon$.
\end{itemize}
\end{defn}

The following result describes a basic property of
$(H)$-normal maps and their behaviour under perturbation.

\begin{prop}\label{prop:normal}
Assume that $(H)$ is maximal in $\Iso(\Omega)$ and $f$ is $(H)$-normal on $U^{\epsilon}$.
Then $\left[f^c_{U,\epsilon}\right]=[\emptyset]$ 
in $\mathcal{F}_G^\nabla\left[\Omega\setminus\Omega_{(H)}\right]$.
\end{prop}

\begin{proof}
By definition, $f^{-1}(0)\subset U\subset\Omega_{(H)}$,
$f=\nabla\varphi$,
where $\varphi(x+v)=\varphi(x)+\frac12\abs{v}^2$ 
for $x+v\in U^{\epsilon}\subset D_f$,
and 
$
f^c_{U,\epsilon}=
\restrictionmap{\nabla\varphi_1}{D_{\varphi_1}\setminus\Omega_{(H)}}
$,
where
\[
\varphi_1(z)=\begin{cases}
\varphi(x)+\frac12\mu^2(\abs{v})\abs{v}^2+\omega(\abs{v})&
\text{if $z=x+v\in U^{\epsilon}$},\\
\varphi(z)&
\text{if $z\in D_f\setminus (U^{\epsilon}\cup B^{\epsilon})$}.
\end{cases}
\]
Define the family of potentials
$h_t\colon D_{\varphi_1}\setminus\Omega_{(H)}\to\R$
\[
h_t(z)=
\begin{cases}
\varphi(x)+\frac12\mu^2(\abs{v})\abs{v}^2+\omega(\abs{v})+t\omega\left(\frac23\abs{v}\right)&
\text{if $z=x+v\in U^{\epsilon}\setminus U$},\\
\varphi(z)&
\text{if $z\in D_f\setminus (U^{\epsilon}\cup 
B^{\epsilon}\cup\Omega_{(H)})$}.
\end{cases}
\]
Observe that
$\nabla h_t$ is a path from $f^c_{U,\epsilon}$ to $\nabla h_1$ in 
$\mathcal{F}_G^\nabla\left(\Omega\setminus\Omega_{(H)}\right)$ and
$\left(\nabla h_1\right)^{-1}(0)=\emptyset$.
Hence $\left[f^c_{U,\epsilon}\right]=[\emptyset]$ 
in $\mathcal{F}_G^\nabla\left[\Omega\setminus\Omega_{(H)}\right]$.
\end{proof}

\section{Orbit-normal maps}\label{sec:onormal}

Here we introduce an important subclass of
$(H)$-normal maps, which will appear in 
the formulation of Main Theorem
as base functions for $\Theta$.
Let $\Or$ denote a $G$-orbit in $\Omega$. Assume that
$\Or^\epsilon:=\{x+v\mid x\in\Or, v\in(T_x\Or)^\bot,\abs{v}<\epsilon\}$
is contained with its closure 
in some tubular neighbourhood of $\Or$ in $\Omega$.

\begin{defn}
A map $f\in\mathcal{F}_G^\nabla(\Omega)$ 
is called \emph{orbit-normal} around $\Or$ if
\begin{itemize}
\item $f^{-1}(0)=\Or$,
\item $\Or^\epsilon\subset D_f$,
\item $f(x+v)=v$ for $x+v\in\Or^\epsilon$.
\end{itemize}
\end{defn}

The three following properties of orbit-normal maps
will be needed in the proof of Main Theorem.
The first one explains the relation between
the notions of orbit-normal and $(H)$-normal maps.

\begin{prop}\label{prop:onormal}
If $(H)$ is an orbit type of an orbit $\Or$ then
every orbit-normal map around $\Or$ is also $(H)$-normal.
\end{prop}

\begin{proof}
Assume that $f\in\mathcal{F}_G^\nabla(\Omega)$ is orbit-normal 
around $\Or$. Let for $x\in\Or$ ($G_x=H$)
\begin{align*}
N_1^x:=&(T_x(\Or))^\bot\cap T_x\Omega_{H},\\
N_2^x:=&(T_x\Omega_{(H)})^\bot.
\end{align*} 
By definition, $N_1^x\perp N_2^x$.
We will show that $(T_x\Or)^\bot=N_1^x\oplus N_2^x$.
It is well-known that
\[
T_x\Omega_{(H)}=
T_x\Or\oplus((T_x\Or)^\bot\cap V^H).
\]
Hence
\begin{multline*}
(T_x(\Or))^\bot\cap T_x\Omega_{(H)}=
(T_x(\Or))^\bot\cap(T_x\Or\oplus((T_x\Or)^\bot\cap V^H))=\\
(T_x\Or)^\bot\cap V^H)=
(T_x(\Or))^\bot\cap T_x\Omega_{H}=N_1^x
\end{multline*}
and, in consequence, 
\[
(T_x\Or)^\bot=
((T_x\Or)^\bot\cap T_x\Omega_{(H)})\oplus(T_x\Omega_{(H)})^\bot
=N_1^x\oplus N_2^x.
\]
Set $U=\{x+v_1\mid x\in\Or, v_1\in N_1^x,
\abs{v_1}<\tfrac{\sqrt{2}}{2}\epsilon\}$.
Since
\[
U^{\frac{\sqrt{2}}{2}\epsilon}
=\{x+v_1+v_2\mid x\in\Or, v_1\in N_1^x,
\abs{v_1}<\tfrac{\sqrt{2}}{2}\epsilon,
v_2\in N_2^x, \abs{v_2}<\tfrac{\sqrt{2}}{2}\epsilon\}
\subset\Or^\epsilon,
\]
we have $f(x+v_1+v_2)=v_1+v_2=f(x+v_1)+v_2$ for 
$x+v_1+v_2\in U^{\frac{\sqrt{2}}{2}\epsilon}$,
which proves that $f$ is $(H)$-normal.
\end{proof}

It turns out that the property of being orbit-normal
is inherited by the restriction to $\Omega_{H}$.

\begin{prop}\label{prop:free}
Assume that $G$-orbit $\Or$ has orbit type $(H)$.
If $f\in\mathcal{F}_G^\nabla(\Omega)$ is orbit-normal 
around $\Or$ then 
$\restrictionmap{f}{D_f\cap\Omega_{H}}
\in\mathcal{F}_{WH}^\nabla(\Omega_H)$
is orbit-normal around $WH$-orbit $\Or\cap\Omega_H$.
\end{prop}

\begin{proof}
The assertion follows from the observation that
\[
\{x+v\mid x\in\Or\cap\Omega_H, v\in V^H, \abs{v}<\epsilon\}=
\Or^\epsilon\cap\Omega_H
\]
is a tubular neighbourhood  of the $WH$-orbit $\Or\cap\Omega_H$,
in which $f(x+v)=v$.
\end{proof}

The next result describes what happens when
we divide out the free action in an orbit-normal map.

\begin{prop}\label{prop:source}
If $G$ acts freely on $\Omega$ and 
$f=\nabla\varphi\in\mathcal{F}_G^\nabla(\Omega)$
is orbit-normal around $\Or$ then $\Or/G\in\Omega/G$
is a source for 
$\nabla\wt\varphi\in\mathcal{F}^\nabla(\Omega/G)$.
\end{prop}

\begin{proof}
Let $p=\Or/G\in\Omega/G=M$. Choose $x\in\Or$.
We can identify some neighbourhood $U$ of $p$ in $M$
with the set $\{v\in(T_x\Or)^\bot\mid\abs{v}<\epsilon\}$
and for $v\in U$ the tangent space $T_vM$ with $(T_x\Or)^\bot$.
Since $\nabla\varphi(x+v)=v$, for $v\in U$ we have
$\nabla\wt\varphi(v)=v\in T_v M$. Hence $p$ is a source for 
$\nabla\wt\varphi$.
\end{proof}

\section{Main results}\label{sec:main}

We can now formulate main results of our paper.
Theorem \ref{thm:phi} will be proved in the next section.

\begin{thm}\label{thm:phi}
The function
\[
\Phi\colon
\mathcal{F}_G^\nabla[\Omega]\to
\prod_{i=1}^m
\mathcal{F}_{WH_i}^\nabla\left[\Omega_{H_i}\right],
\]
where the product is taken over $\Iso(\Omega)$,
is a bijection. Moreover,
if $f$ is $(H_j)$-normal then 
\[
\Phi_i([f])
=\begin{cases}
[\emptyset]& \text{if $i\neq j$},\\
[f_{H_j}]& \text{if $i=j$},
\end{cases}
\]
where $f_{H_j}=\restrictionmap{f}{D_f\cap\Omega_{H_j}}$.
\end{thm}

\begin{main}
Assume that $0\not\in\Omega$ or $\dim V^{G}>0$.
Then the function
\[
\Theta\colon
\mathcal{F}_G^\nabla[\Omega]\to
\prod_{i=1}^m\biggl(\sum_{j}\mathbb{Z}\biggr),
\]
where the product is taken over $\Iso(\Omega)$
and the respective direct sums are indexed by
either finite or countably infinite
sets of connected components
of the quotients $\Omega_{H_i}/WH_i$,
is a bijection. Moreover
\begin{enumerate}
\item $\Theta([f\sqcup g])=\Theta([f])+\Theta([g])$
      for $f,g\in\mathcal{F}_G^\nabla(\Omega)$ such that
			$D_f\cap D_g=\emptyset$,
\item $\Theta([\emptyset])=0$,	
\item if $\Theta([f])\neq0$ then there is $x\in D_f$
      such that $f(x)=0$,		
\item if $f$ is orbit-normal
      around $\Or\subset\Omega_{(H_k)}$
			and 
			$p_k\left(\Or\cap\Omega_{H_k}\right)\in(\Omega_{H_k}/WH_k)_l$,
			where $p_{k}\colon\Omega_{H_k}\to\Omega_{H_k}/WH_k$ denotes 
            the quotient map and $(\Omega_{H_k}/WH_k)_l$ 
            the respective component,
			then
      \[\Theta_{ij}([f])=\begin{cases}
        1& \text{if $i=k$ and $j=l$},\\
        0& \text{otherwise}.
        \end{cases}
      \]
\end{enumerate}
\end{main}

\begin{proof}
First we show that $\Theta$ is a bijection.
Define 
\[
\Psi=\prod_{i=1}^m\Psi_i\colon\prod_{i=1}^m\Psi_i
\mathcal{F}_{WH_i}^\nabla\left[\Omega_{H_i}\right]\to
\prod_{i=1}^m\mathcal{F}^\nabla\left[\Omega_{H_i}/WH_i\right]
\]
and
\[
\pi=\sum_{i,j}\pi_{ij}\colon
\prod_{i=1}^m\mathcal{F}^\nabla\left[\Omega_{H_i}/WH_i\right]\to
\prod_{i=1}^m\sum_j\mathcal{F}^\nabla\left[(\Omega_{H_i}/WH_i)_j\right],
\]
where for each $i$ the index $j$ runs over the set of connected 
components of $\Omega_{H_i}/WH_i$.
Observe that $\Psi$ and $\pi$ are bijections 
(the first from Theorem \ref{thm:psi} and the second by definition).
Let 
$\I_{ij}\colon\mathcal{F}^\nabla\left[(\Omega_{H_i}/WH_i)_j\right]\to\Z$
denote the intersection number restricted to the respective component.
By Remark \ref{rem:gradient}, $\I_{ij}$ is a bijection and, in consequence,
so is 
$
\I=\sum_{i,j}\I_{ij}\colon
\prod_{i=1}^m\sum_j\mathcal{F}^\nabla\left[(\Omega_{H_i}/WH_i)_j\right]
\to\prod_{i=1}^m\sum_j\Z
$. 
Since, by Theorem \ref{thm:phi}, $\Phi$ is a bijection, we obtain that
the composition $\Theta=\I\circ\,\pi\circ\Psi\circ\Phi$
is also a bijection.

Next we prove the additivity property (1).
Following the notation from Remark \ref{rem:defin}
we obtain the sequences 
$f_i, g_i, (f\sqcup g)_i\in\mathcal{F}_G^\nabla(\Omega_i)$ and 
$f_i', g_i', (f\sqcup g)_i'=f_i'\sqcup g_i'
\in\mathcal{F}_{WH_i}^\nabla\left(\Omega_{H_i}\right)$.
Since $\Phi_i([f\sqcup g])=\left[f_i'\sqcup g_i'\right]$, 
by Corollary \ref{cor:psi} we have
\begin{multline*}
\Theta_{ij}([f\sqcup g])=
\I\circ\,\pi_{ij}\circ\Psi_i\left(\left[f_i'\sqcup g_i'\right]\right)=\\
\I\circ\,\pi_{ij}\circ\Psi_i\left(\left[f_i'\right]\right)+
\I\circ\,\pi_{ij}\circ\Psi_i\left(\left[g_i'\right]\right)=
\Theta_{ij}([f])+\Theta_{ij}([g]).
\end{multline*}

The property (2) follows from (1) as well as from the direct construction.

To prove (3) observe that if $f^{-1}(0)=\emptyset$ then
$f$ is otopic to the empty map and hence
$\Theta_{ij}([f])=\Theta_{ij}([\emptyset])=0$.
 
Finally, we show the normalization property (4).
By Proposition \ref{prop:onormal}, $f$ is $(H_k)$-normal.
First consider the case $i=k$. Note that
$\Phi_k([f])=\left[f_{H_k}\right]$ and
$\Psi_k(\left[f_{H_k}\right])=
\Psi_k(\left[\nabla\varphi\right])=\left[\nabla\wt\varphi_{k}\right]$.
By Propositions \ref{prop:free} and \ref{prop:source},
the only zero of $\nabla\wt\varphi_{k}$ is a source
and, by assumption, it is contained in $(\Omega_{H_k}/WH_k)_l$. 
Consequently,
\[
\Theta_{kj}([f])=
\I\circ\,\pi_{kj}\circ\Psi_k\left(\left[f_{H_k}\right]\right)=
\I\circ\,\pi_{kj}\left(\left[\nabla\wt\varphi_{k}\right]\right)=
\begin{cases}
1&\text{for $k=l$,}\\
0&\text{for $k\neq l$.}
\end{cases}
\]
In turn, for $i\neq k$, $\Phi_i([f])=\left[\emptyset\right]$ and,
in consequence, $\Theta_{ij}([f])=0$. 
\end{proof}	

\begin{rem}\label{rem:opposite}
Now consider the case $0\in\Omega$ and $\dim V^G=0$.
In that situation $H_1=G$, $V^{H_1}=\Omega_{H_1}=\{0\}$,
and $WH_1$ is trivial.
Since
\[
\Psi_1\circ\Phi_1\colon
\mathcal{F}_G^\nabla\left[\Omega\right]\to
\mathcal{F}^\nabla\left[\Omega_{H_1}/WH_1\right]=
\mathcal{F}^\nabla\left[\{0\}\right]=\{\emptyset,0\},
\]
we can express $\Theta$ the in following form
\[
\Theta\colon
\mathcal{F}_G^\nabla[\Omega]\to
\{0,1\}\times
\prod_{i=2}^m\sum_{j}\mathbb{Z},
\]
where $\Theta_{11}\colon\mathcal{F}_G^\nabla[\Omega]\to\{0,1\}$ and
\[
\Theta_{11}([f])=
\begin{cases}
1&\text{for $0\in D_f$,}\\
0&\text{for $0\not\in D_f$.}
\end{cases}
\]
Main Theorem holds as well in the above case.
Regarding the additivity property
let us mention that all $\Z$ have complete
additive structure, but the addition $1+1$ in the set $\{0,1\}$
makes no sense. Nevertheless, when $D_f\cap D_g=\emptyset$ 
then either $\Theta_{11}([f])$ or $\Theta_{11}([g])$ 
is equal to $0$ and therefore the condition (1) makes no problem.
\end{rem}

\section{Proof of Theorem  \texorpdfstring{\ref{thm:phi}}{7.1}}\label{sec:proof}

This section contains the proof of Theorem \ref{thm:phi}
preceded by a series of lemmas and notations.
Let us assume that $(H)$ is maximal in  $\Iso(\Omega)$.
Recall that below $\sim$ denotes the relation of gradient otopy.

\begin{lem}
$\wh{\Phi}$ is well-defined.
\end{lem}

\begin{proof}
Observe that the definition of $\wh\Phi$ does not depend on the choice
of $U$ and $\epsilon$ if only they satisfy the condition \eqref{eq:def}.
Since for the fixed $U$ the definition
of $\wh\Phi$ does not depend on the choice of $\epsilon$,
it remains to check that it
does not depend on the choice of $U$ if $\epsilon$ is fixed.
Let $\varphi_1$ and $\varphi_1'$ be potentials from the definition
of $\wh\Phi$ corresponding to $U$ and $V$.
We can assume that $U\subset V$ because otherwise
we can pass from $U$ and $V$ through $U\cap V$.
Our assertion follows from the observation that
\[
\restrictionmap{\nabla\varphi_1'}{D_{\varphi_1'}\setminus\Omega_{(H)}}
\sim
\restrictionmap{\nabla\varphi_1'}{A}
\sim
\restrictionmap{\nabla\varphi_1}{A}
\sim
\restrictionmap{\nabla\varphi_1}{D_{\varphi_1}\setminus\Omega_{(H)}}
\]
where 
$A=
D_f\setminus
\left(\Omega_{(H)}\cup B^{\epsilon}_U\cup B^{\epsilon}_V\right)$.

Now we will show that
if $[f]=[g]$ in 
$\mathcal{F}^\nabla_G[\Omega]$
then 
\begin{enumerate}
\item $\wh{\Phi}'([f])=[f_H]=[g_H]=\wh{\Phi}'([g])$ in
      $\mathcal{F}^\nabla_{WH}\left[\Omega_H\right]$,
\item $\wh{\Phi}''([f])=\wh{\Phi}''([g])$  
      in $\mathcal{F}^\nabla_G\left[\Omega\setminus\Omega_{(H)}\right]$.
\end{enumerate}
By assumption, there is an otopy 
$h\colon\Lambda\subset I\times\Omega\to V$
such that $f=h_0$ and $g=h_1$.
The proof of (1) is straightforward.
Namely, let $A=\Lambda\cap(I\times\Omega_H)$
and $k=\restrictionmap{h}{A}$.
Then $k\colon A\subset I\times\Omega_H\to V^H$
is an otopy such that $k_0=f_H$ and $k_1=g_H$,
and (1) is proved. 
To show (2) we will perturb the otopy $h$
treating every section $h_t(\cdot)=h(t,\cdot)$ analogously to 
the perturbation $f_{U,\epsilon}$ of the map $f$.
Let $W$ be an open and invariant subset of 
$I\times\Omega_{(H)}$ such that
\[
h^{-1}(0)\cap\left(I\times \Omega_{(H)}\right)\subset
W\subset\cl W\Subset\Lambda\cap\left(I\times\Omega_{(H)}\right).
\]
Let $B=\bd W$. For $\epsilon>0$ let us define the sets
\begin{align*}
W^\epsilon=&\{(t,x+v)\mid (t,x)\in W, v\in N_x, \abs{v}<\epsilon\},\\
B^\epsilon=&\{(t,x+v)\mid (t,x)\in B, v\in N_x, \abs{v}\le\epsilon\}.
\end{align*}
Observe that for $\epsilon$ sufficiently small
we have 
$W^{\epsilon}\subset\Lambda$,
$\cl(W^{\epsilon})$ is contained in some tubular neighbourhood 
of $I\times\Omega_{(H)}$ and 
$h^{-1}(0)\cap B^{\epsilon}=\emptyset$.
Recall that for $X\subset I\times\Omega$
we denote by $X_t$ the set 
$\{x\in\Omega\mid(t,x)\in A\}$.
Define the map
\[
h_{W,\epsilon}\colon\Lambda\setminus B^{\epsilon}\to V
\]
by the formula
\[
h_{W,\epsilon}(z)=\begin{cases}
\left(h_t\right)_{W_t,\epsilon}(x)&
\text{if $z=(t,x)\in W^{\epsilon}$},\\
h(z)&
\text{if $z\in\Lambda\setminus(W^{\epsilon}\cup 
B^{\epsilon})$}.
\end{cases}
\]
In the above formula we use notation of perturbation
introduced in Subsection \ref{subsec:def}.
Set
\[
h^c_{W,\epsilon}=
\restrictionmap{h_{W,\epsilon}}
{\Lambda\setminus\left(B^{\epsilon}\cup\left(I\times\Omega_{(H)}\right)\right)}.
\]
Since $h_{W,\epsilon}$ is an otopy 
in $\mathcal{F}^\nabla_{G}\left(\Omega\right)$
and $\left(h^{c}_{W,\epsilon}\right)^{-1}(0)$ is compact,
$h^c_{W,\epsilon}$ is an otopy in
$\mathcal{F}^\nabla_G\left(\Omega\setminus\Omega_{(H)}\right)$ 
connecting 
$\restrictionmap{f^c_{W_0,\epsilon}}{D_f\setminus B_0^{\epsilon}}$ and 
$\restrictionmap{g^c_{W_1,\epsilon}}{D_g\setminus B_1^{\epsilon}}$.
Consequently,
\[
\left[f^c_{W_0,\epsilon}\right]=
\left[\restrictionmap{f^c_{W_0,\epsilon}}{D_f\setminus B_0^{\epsilon}}\right]=
\left[\restrictionmap{g^c_{W_1,\epsilon}}{D_g\setminus B_1^{\epsilon}}\right]=
\left[g^c_{W_1,\epsilon}\right]\quad
\text{in $\mathcal{F}^\nabla_{G}\left[\Omega\setminus\Omega_{(H)}\right]$},
\]
which gives $\wh{\Phi}''([f])=\wh{\Phi}''([g])$ and (2) is proved.
\end{proof}

The following two constructions will be needed
in the proof of Lemma \ref{lem:bij}.

Assume $k=\nabla\varphi\in\mathcal{F}_{WH}^\nabla\left(\Omega_H\right)$.
Let $\varphi_G\colon GD_k\to\R$ be given by $\varphi_G(gx)=\varphi(x)$.
Let $U$ be an open bounded invariant subset of $GD_k$ such that
$k^{-1}(0)\subset U$. Define the function
$\wt\varphi\colon U^\epsilon\to\R$ by
$\wt\varphi(x+v)=\varphi_G(x)+\frac12\abs{v}^2$ for $x+v\in U^\epsilon$.
Set 
\[
k^{U,\epsilon}=\nabla\wt\varphi.
\]
For $l\in\mathcal{F}_G^\nabla\left(\Omega\setminus\Omega_{(H)}\right)$
and $Y\subset\Omega$ closed invariant such that
$l^{-1}(0)\cap Y=\emptyset$ we define
\[
l^{Y}=
\restrictionmap{l}{D_l\setminus Y}.
\]

In the following proposition we use the notation
introduced in Subsection \ref{subsec:def}.

\begin{prop}\label{prop:constr}
The maps $k^{U,\epsilon}\in\mathcal{F}_G^\nabla\left(\Omega\right)$
and 
$l^{Y}\in\mathcal{F}_G^\nabla\left(\Omega\setminus\Omega_{(H)}\right)$
have the following properties:
\begin{enumerate}
\item $k^{U,\epsilon}$ is $(H)$-normal,\vspace{0.5mm}
\item in $\mathcal{F}_G^\nabla\left[\Omega\setminus\Omega_{(H)}\right]$ 
      we have\vspace{0.5mm}
\begin{enumerate}
\item $\left[\left(k^{U,\epsilon}\right)^c_{U,\epsilon}\right]=\left[\emptyset\right]$,\vspace{0.5mm}
\item $\left[\left(k^{U,\epsilon}\sqcup l^{\cl\left(U^\epsilon\right)}\right)^c_{U,\epsilon}\right]
      =\left[l^{\cl\left(U^\epsilon\right)}\right]$,\vspace{0.5mm}
\item $\left[l^{\cl\left(U^\epsilon\right)}\right]=\left[l\right]$,
\end{enumerate}\vspace{0.5mm}
\item $\left(f^c_{U,\epsilon}\right)^{\cl\left(U^{\epsilon/3}\right)}=f^a_{U,\epsilon}$,\vspace{0.5mm}
\item $\left(f_H\right)^{U,\epsilon/3}=\restrictionmap{f_{U,\epsilon}}{U^{\epsilon/3}}=f^n_{U,\epsilon}$.
\end{enumerate}
\end{prop}

\begin{proof}
Property (2a) follows from (1) and Proposition \ref{prop:normal}.
All other properties are obvious.
\end{proof}

Let us define the function
\[
\wt{\Phi}\colon
\mathcal{F}_{WH}^\nabla\left[\Omega_H\right]\times
\mathcal{F}_G^\nabla\left[\Omega\setminus\Omega_{(H)}\right]\to
\mathcal{F}_G^\nabla[\Omega]
\]
by the formula
\[
\wt{\Phi}\left([k],[l]\right)=
\left[k^{U,\epsilon}\sqcup l^{\cl\left(U^\epsilon\right)}\right].
\]

It will turn out that $\wt{\Phi}$ is inverse to $\wh{\Phi}$,
which will imply that $\wh{\Phi}$ is a bijection.

\begin{lem}
$\wt{\Phi}$ is well-defined.
\end{lem}

\begin{proof}
It is easy to see that the definition of $\wt{\Phi}$
does not depend on the choice of $U$ and $\epsilon$.
We show that $\wt{\Phi}$ is also independent of 
the choice of the representative in the otopy class.
Let $k\colon D_k\subset I\times\Omega_H\to V^H$
and 
$l\colon D_l\subset I\times\left(\Omega\setminus\Omega_{(H)}\right)\to V$
be otopies. Let $W\subset I\times\Omega_{(H)}$ 
be an open invariant subset such that
\[
k^{-1}(0)\subset W\subset\cl W\Subset GD_k.
\]
Choose $\epsilon>0$ such that 
$l^{-1}(0)\cap\cl\left(W^\epsilon\right)=\emptyset$.
Recall that for $X\subset I\times\Omega$ we denote by $X_t$
the set $\{x\in\Omega\mid(t,x)\in X\}$.
Since
\[
h_t:=\left(k_t\right)^{W_t,\epsilon}
\sqcup \left(l_t\right)^{\left(\cl\left(W^\epsilon\right)\right)_t}
\]
is an otopy and
\[
\left(k_i\right)^{W_i,\epsilon}
\sqcup \left(l_i\right)^{\cl\left(W_i^\epsilon\right)}
\sim
\left(k_i\right)^{W_i,\epsilon}
\sqcup \left(l_i\right)^{\left(\cl\left(W^\epsilon\right)\right)_i}
\qquad\text{for $i=0,1$}
\]
we obtain $\wt{\Phi}\left([k_0],[l_0]\right)=\wt{\Phi}\left([k_1],[l_1]\right)$,
which is the desired conclusion.
\end{proof}

\begin{lem}\label{lem:bij}
$\wt{\Phi}$ is inverse to $\wh{\Phi}$
and therefore
$\wh{\Phi}$ is a bijection.
\end{lem}

\begin{proof}
The calculations below are based on Proposition \ref{prop:constr}.
Observe that
\begin{multline*}
\wh{\Phi}\circ\wt{\Phi}\left([k],[l]\right)=
\wh{\Phi}\left(\left[k^{U,\epsilon}\sqcup l^{\cl\left(U^\epsilon\right)}\right]\right)=\\
\left(\left[\restrictionmap{k}{U}\right],
\left[\left(k^{U,\epsilon}\sqcup l^{\cl\left(U^\epsilon\right)}\right)^c_{U,\epsilon}\right]\right)=
\left([k],\left[l^{\cl\left(U^\epsilon\right)}\right]\right)=
\left([k],[l]\right).
\end{multline*}
In turn, if in $\wt{\Phi}$ we take $\epsilon/3$ instead of $\epsilon$ we obtain
\[
\wt{\Phi}\circ\wh{\Phi}\left([f]\right)=
\wt{\Phi}\left(\left[f_H\right],\left[f^c_{U,\epsilon}\right]\right)=
\left[\left(f_H\right)^{U,\epsilon/3}\sqcup
\left(f^c_{U,\epsilon}\right)^{\cl\left(U^{\epsilon/3}\right)}\right]=
\left[f^n_{U,\epsilon}\sqcup
f^a_{U,\epsilon}\right]=[f],
\]
which completes the proof.
\end{proof}

The next lemma follows directly from the construction
of $\wh{\Phi}$ and the definition of the sequence $f_i$
(Remark \ref{rem:defin}).
\begin{lem}\label{lem:phii}
Let $f\in\mathcal{F}_G^\nabla(\Omega)$.
For each $i$ the bijection $\wh{\Phi}_i$
has the following properties:
\begin{enumerate}
\item $\wh{\Phi}_i([\emptyset])=([\emptyset],[\emptyset])$,
\item $\wh{\Phi}''_i([f_{i}])=[f_{i+1}]$,
\item $\wh{\Phi}'_i([f_{i}])=\Phi_i([f])$.
\end{enumerate}
\end{lem}

Now we can move on to the final purpose of this section.
\begin{proof}[Proof of Theorem \ref{thm:phi}]
First we show that $(\Phi_1,\Phi_2,\dotsc,\Phi_m)$ is a bijection. 
Since
\begin{enumerate}
\item $(\Phi_1,\Xi_1)=\wh{\Phi}_1$ is a bijection,
\item the bijectivity of $(\Phi_1,\dotsc,\Phi_i,\Xi_i)$
      and $\wh{\Phi}_{i+1}$
      imply the bijectivity of 
      \[(\Phi_1,\dotsc,\Phi_i,\Phi_{i+1},\Xi_{i+1})
      =(\Phi_1,\dotsc,\Phi_i,\wh{\Phi}_{i+1}\circ\Xi_{i}),\]
\item the set of values of $\Xi_m$ is  equal to the singleton
      $\mathcal{F}_G^\nabla[\emptyset]$,     
\end{enumerate}
by induction on $i$ we obtain  
that $(\Phi_1,\Phi_2,\dotsc,\Phi_m)$ is 
a bijection.

Now we will prove the second part of our statement.
Assume that $f$ is $(H_j)$-normal.
Note that $f^{-1}(0)\subset\Omega_{(H_j)}$.
First observe that since $f_i=\restrictionmap{f}{D_f\cap\Omega_i}$
for $i\le j$ we have
\[
\Phi_i([f])=
[\restrictionmap{f}{D_f\cap\Omega_{H_i}}]=
\begin{cases}
[\emptyset]& \text{if $i<j$},\\
[f_{H_j}]& \text{if $i=j$}.
\end{cases}
\]
Note that $(H_j)$ is maximal in $\Iso(\Omega_j)$.
By definition, 
$f_j=\restrictionmap{f}{D_f\cap\Omega_j}$
and $f_{j+1}=(f_j)^c_{U,\epsilon}$.
Hence, by Proposition \ref{prop:normal}, $[f_{j+1}]=[\emptyset]$ in 
$\mathcal{F}_G^\nabla\left[\Omega_{j+1}\right]$.
By Lemma \ref{lem:phii},
$[f_i]=[\emptyset]$ in 
$\mathcal{F}_G^\nabla\left[\Omega_i\right]$ for $i>j$,
and, in consequence,
\[
\Phi_i([f])=
\wh{\Phi}'_i([f_{i}])
=[\emptyset]\qquad\text{for $i>j$.}
\]
\end{proof}

\section{The Parusi\'nski type theorem for equivariant local maps}
\label{sec:parusinski}

The aim of this section is to compare the sets of equivariant
and equivariant gradient otopy classes.
Note that the inclusion
$\mathcal{F}_G^{\nabla}(\Omega)\hookrightarrow\mathcal{F}_G(\Omega)$
induces the well-defined function 
$\mathcal{J}\colon\mathcal{F}_G^{\nabla}[\Omega]\to\mathcal{F}_G[\Omega]$.
In \cite{BP1} we proved that in the absence of group action
the function $\mathcal{J}$ is a bijection.
It turns out that in the equivariant case
the following Parusi\'nski type theorem (see \cite{P}) holds.
Define~$\Iso_0(\Omega):=\{(H)\in\Iso(\Omega)\mid\dim WH=0\}$.

\begin{thm}\label{thm:par}
$\mathcal{J}$ is a bijection if and only if 
$\Iso_0(\Omega)=\Iso(\Omega)$.
\end{thm}

The remainder of this section will be devoted to the proof of this result.
Recall that in \cite{B1} we showed the existence of the bijection
$
\Upsilon\colon\mathcal{F}_G[\Omega]\to
\prod_{i=1}^m\mathcal{F}_{WH_i}\left[\Omega_{H_i}\right]
$
and in the present paper the existence of the bijection
$
\Phi\colon\mathcal{F}_G^\nabla[\Omega]\to
\prod_{i=1}^m\mathcal{F}_{WH_i}^\nabla\left[\Omega_{H_i}\right]
$.
It is natural to consider the following diagram
\begin{equation}\label{diag:prod}
\begin{CD}
\mathcal{F}_G^{\nabla}[\Omega]@>\Phi>>
\prod_{i=1}^m\mathcal{F}_{WH_i}^\nabla\left[\Omega_{H_i}\right]\\
@V\mathcal{J}VV @VV\prod\mathcal{J}_iV\\
\mathcal{F}_G[\Omega]@>\Upsilon>>
\prod_{i=1}^m\mathcal{F}_{WH_i}\left[\Omega_{H_i}\right],
\end{CD}
\end{equation}
where
$\mathcal{J}_i\colon\mathcal{F}_{WH_i}^\nabla\left[\Omega_{H_i}\right]
\to\mathcal{F}_{WH_i}\left[\Omega_{H_i}\right]$ 
are also induced by the inclusions.
The commutativity of diagram \eqref{diag:prod} follows from
the inductive definitions of $\Phi$ and $\Upsilon$ and
the following result. 

\begin{lem}
Assume that $(H)$ is maximal in $\Iso(\Omega)$.
Let 
$
\mathcal{J}'\colon
\mathcal{F}_{WH}^\nabla\left[\Omega_H\right]
\to\mathcal{F}_{WH}\left[\Omega_H\right]
$
and
$
\mathcal{J}''\colon
\mathcal{F}_G^\nabla\left[\Omega\setminus\Omega_{(H)}\right]
\to\mathcal{F}_G\left[\Omega\setminus\Omega_{(H)}\right]
$
be induced by the respective inclusions.
Then the diagram
\begin{equation}\label{diag:times}
\begin{CD}
\mathcal{F}_G^{\nabla}[\Omega]@>\wh\Phi>>
\mathcal{F}_{WH}^\nabla\left[\Omega_H\right]\times
\mathcal{F}_G^\nabla\left[\Omega\setminus\Omega_{(H)}\right]\\
@V\mathcal{J}VV @VV\mathcal{J}'\times\mathcal{J}''V\\
\mathcal{F}_G[\Omega]@>\wh\Upsilon>>
\mathcal{F}_{WH}\left[\Omega_H\right]\times
\mathcal{F}_G\left[\Omega\setminus\Omega_{(H)}\right]
\end{CD}
\end{equation}
commutes.
\end{lem}

\begin{proof}
Let $f=\nabla\varphi\in\mathcal{F}_G^{\nabla}(\Omega)$.
Recall that $\wh{\Phi}([f])=
\left(\left[f_H\right],\left[k\right]\right)$, where
\[
k(z)=\begin{cases}
\nabla\left(\varphi\circ r_1\right)(z)+\nabla\omega(\abs{v})&
\text{if $z=x+v\in U^{\epsilon}\setminus\Omega_{(H)}$},\\
\nabla\varphi(z)&
\text{if $z\in D_f\setminus\left(U^\epsilon\cup 
B^{\epsilon}\cup\Omega_{(H)}\right)$},
\end{cases}
\] 
and $\wh{\Upsilon}([f])=
\left(\left[f_H\right],\left[\wt k\right]\right)$, where
\[
\wt k(z)=\begin{cases}
\nabla\varphi\left(r_1(z)\right)+\nabla\omega(\abs{v})&
\text{if $z=x+v\in U^{\epsilon}\setminus\Omega_{(H)}$},\\
\nabla\varphi(z)&
\text{if $z\in D_f\setminus\left(U^\epsilon\cup 
B^{\epsilon}\cup\Omega_{(H)}\right)$}.
\end{cases}
\] 
First observe that 
$\mathcal{J}'\circ\wh{\Phi}'([f])=\left[f_H\right]=
\wh{\Upsilon}'\circ\mathcal{J}([f])$. To prove that
$\mathcal{J}''\circ\wh{\Phi}''=\wh{\Upsilon}''\circ\mathcal{J}$
it is enough to show that the straight-line homotopy
$h_t=(1-t)k+t\wt k$ is an otopy in 
$\mathcal{F}_G\left(\Omega\setminus\Omega_{(H)}\right)$.
To see that, note that
\[
k(z)=\nabla\left(\varphi\circ r_1\right)(z)+\nabla\omega(\abs{v})=
\left(Dr_1(z)\right)^T\cdot\nabla\varphi(r_1(z))+\nabla\omega(\abs{v})
\quad\text{on $ U^{\epsilon}\setminus\Omega_{(H)}$}.
\]
Using orthogonal coordinates $x$ being principal directions
corresponding to the principal curvatures of $\Omega_{(H)}$ and
orthogonal coordinates $v$ we get that the matrix
$\left(Dr_1(z)\right)^T=\left(D(x+\mu(\abs{v})v)\right)^T$ 
has the following form
\[
\left(
\begin{array}{c:c} A &  0 \\ \hdashline
0 & B \end{array}
\right),
\]
where $A$ is diagonal with positive entries on the diagonal
and $B=D(\mu(\abs{v})v)$ is nonsingular if{f} 
$2\epsilon/3<\abs{v}<\epsilon$. Therefore
\[
h_t(z)=0\equiv
k(z)=0\equiv
\wt k(z)=0
\]
for all $t\in I$ and, in consequence, $h$ is an otopy.
\end{proof}

Since from Diagram \eqref{diag:prod} $\mathcal{J}$ is a bijection
if and only if all $\mathcal{J}_i$ are bijections
and the action of $WH_i$ on $\Omega_{H_i}$ is free,
it remains to study the function
${\mathcal{J}}\colon\mathcal{F}_G^{\nabla}
[\Omega]\to\mathcal{F}_G[\Omega]$
assuming $G$ acts freely on $\Omega$. We will consider two cases.
If $\dim G>0$ then $\mathcal{F}_G[\Omega]$ \
is trivial (see \cite[Thm 3.1]{B1}))
and $\mathcal{F}_G^{\nabla}[\Omega]$ is nontrivial 
by Corollary \ref{cor:psi},
and so ${\mathcal{J}}$ is not a bijection.
Now consider the case $\dim G=0$.
Recall that $M=\Omega/G$, $E=(\Omega\times V)/G$
and $\Gamma[M,E]$ denotes the set of otopy classes 
of local cross sections of the bundle $E\to M$. 
The following commutative diagram \eqref{diag:diamond}
relates the sets of different otopy classes.
Since $a$, $b$, $c$, $d$ are bijections by
Theorem \ref{thm:psi}, \cite[Thm 3.4]{B1}, \cite[Thm 5.1]{BP4}
and the natural identification of vector bundles
$E$ and $TM$, so is ${\mathcal{J}}$.

\begin{equation}\label{diag:diamond}
\begin{split}
\xymatrix{
\mathcal{F}_G^{\nabla}[\Omega]\ar@{-->}[dd]_{\mathcal{J}} 
\ar[r]^a &\mathcal{F}^\nabla[M]\ar[d]^c\\
                     &\mathcal{F}[M]\\
\mathcal{F}_G[\Omega]\ar[r]^b &\Gamma[M,E]\ar@{<->}[u]_d}
\end{split}
\end{equation}\vspace{0.5mm}

The commutativity of Diagram \eqref{diag:prod} and the above
considerations concerning the single function
$\mathcal{J}_i\colon\mathcal{F}_{WH_i}^\nabla\left[\Omega_{H_i}\right]
\to\mathcal{F}_{WH_i}\left[\Omega_{H_i}\right]$ 
complete the proof of Theorem \ref{thm:par}.

\begin{rem}
It may be worth noting the relation 
of our constructions with the additive subgroups
$U(V)$ and $A(V)$ of the Euler ring $U(G)$
and the Burnside ring $A(G)$ taking into account
only their additive structure. Namely,
denoting by $\Theta_U$ the function $\Theta$
from Main Theorem, by $\Theta_A$ its nongradient
version from \cite{B1}, by $\pi_U$ and $\pi_A$
the summing on $j$ and by $\mathcal{U}$ the forgetful functor,
we obtain the following commutative diagram:
\begin{equation}\label{diag:deg}
\begin{CD}
\mathcal{F}_G^{\nabla}[\Omega]@>\Theta_U>>
{\displaystyle\prod_{\Iso(\Omega)}\biggl(\sum_j\Z\biggr)}@>\pi_U>>
U(V)={\displaystyle\prod_{\Iso(\Omega)}\Z}\\
@V\mathcal{J}VV @VV\mathcal{U}V @VV\mathcal{U}V\\
\mathcal{F}_G[\Omega]@>\Theta_A>>
{\displaystyle\prod_{\Iso_0(\Omega)}\biggl(\sum_j\Z\biggr)}@>\pi_A>>
A(V)={\displaystyle\prod_{\Iso_0(\Omega)}\Z}.
\end{CD}
\end{equation}

\noindent Moreover, observe that $\pi_U\circ\Theta_U=\deg_G^\nabla$ and
$\pi_A\circ\Theta_A=\deg_G$ using the notation of 
the respective degrees from \cite{BGI}.
\end{rem}

\section{Parametrized equivariant gradient local maps}
\label{sec:parametrized}

We close the paper with some remarks concerning the parametrized case.
Recall that $\R^k$ denotes a trivial representation of $G$.
Let $\Omega$ be an open invariant subset of $\R^k\oplus V$.
Define $\mathcal{F}_G(\Omega)=\Loc_G(\Omega,V,0)$.
Let $(y,x)$ denotes the coordinates $\R^k\oplus V$.
A map $f\in\mathcal{F}_G(\Omega)$ is called \emph{gradient}
if there is a function (not necessarily continuous)
$\varphi\colon D_f\to\R$ such that 
$\varphi$ is $C^1$ with respect to $x$ and 
$f(x,y)=\nabla_x\varphi(x,y)$. 
Define 
$\mathcal{F}_G^\nabla(\Omega)
=\{f\in\mathcal{F}_G(\Omega)\mid\text{$f$ is gradient}\,\}$.
Similarly as in the nonparametrized  case we define 
gradient otopies and the set of gradient otopy classes
$\mathcal{F}_G^\nabla[\Omega]$.

It is easily seen that 
Theorem \ref{thm:phi} also holds
in that case,
because both the construction of 
the function
$
\Phi\colon\mathcal{F}_G^\nabla[\Omega]\to
\prod_{i=1}^m\mathcal{F}_{WH_i}^\nabla\left[\Omega_{H_i}\right]
$
and the proof of Theorem \ref{thm:phi}
are essentially the same.

Now let us describe the single factor of the above decomposition.
Assume that $G$ acts freely on $\Omega\subset\R^k\oplus V$.
Set $M=\Omega/G$. Define $\pi\colon M\to\R$ by $\pi([y,x])=y$.
Observe that $M_y=\pi^{-1}(y)$ is a submanifold of $M$
(possibly empty) for every $y\in\R^k$.
Let us consider a subbundle $E\subset TM$ defined by
$E_p=T_p M_{\pi(p)}$ for $p\in M$.
The set of all local continuous sections 
of the bundle $E$
will be denoted by $\Gamma(M,E)$.
A local section $s\colon D_s\subset M\to E$
is called \emph{gradient} if there is a function 
$\varphi\colon D_s\to\R$ (not necessarily continuous)
such that $\varphi$ is $C^1$ on every $D_s\cap M_y$ and
\[
s(p):=\nabla_\pi\varphi(p)=
\nabla
\left(\restrictionmap{\varphi}{D_s\cap  M_{\pi(p)}}\right)(p).
\]
Write 
$\mathcal{F}_\pi^\nabla(M)
:=\{s\in\Gamma(M,E)\mid\text{$s$ is gradient}\,\}$.
Recall that if $\varphi\colon U\to\R$ 
is an invariant function
then $\wt\varphi$ stands 
for the quotient function
$\wt\varphi\colon U/G\to\R$.
Finally, let 
$\Psi\colon\mathcal{F}^\nabla_G(\Omega)
\to\mathcal{F}_\pi^\nabla(M)$
be defined by 
$\Psi\left(\nabla_x\varphi\right)=
\nabla_\pi\wt\varphi$.
The following result,
which is an analogue of Theorem \ref{thm:psi},
allows us to replace
single factors of our decomposition $\Phi$
by the sets of gradient otopy classes
of parametrized maps on quotient manifolds
(the free $G$-action has been divided out).

\begin{prop}
$\Psi$ is a bijection and induces a bijection between 
the sets of gradient otopy classes
$\mathcal{F}^\nabla_G[\Omega]$ and 
$\mathcal{F}_\pi^\nabla[M]$.
\end{prop}

\begin{rem}
It is worth pointing out that for now
we do not have a satisfactory classification
of the set $\mathcal{F}_\pi^\nabla[M]$
even in the simplest case 
of trivial action, where $M=\Omega=\R^k\oplus\R^n$
(see Question 1 in \cite{BP3}).
\end{rem}

\end{document}